\documentclass[12pt,a4paper]{article}

\usepackage[english]{babel}

\usepackage{geometry}
\usepackage[T1]{fontenc}
\usepackage{relsize}
\usepackage[T1]{fontenc}
\usepackage[utf8]{inputenc} 
\usepackage{lmodern}
\usepackage{bigints}
\usepackage{relsize}
\usepackage{amssymb}
\usepackage[utf8]{inputenc}
\usepackage{amsmath}
\usepackage{amsthm}
\usepackage{graphicx}
\usepackage[colorlinks,linkcolor=black, citecolor=black]{hyperref}
\numberwithin{equation}{section}
\newtheorem{dfn}{Definition}[section]

\newtheorem{thm}{Theorem}[section]
\newtheorem{prop}{Proposition}[section]
\newtheorem{cor}{Corollary}[section]

\newtheorem{lem}{Lemma}[section]
\newtheorem{rem}{Remark}[section]

\newcommand{\bt}{\begin{thm}}
	\newcommand{\et}{\end{thm}}
\newcommand{\bl}{\begin{lem}}
	\newcommand{\el}{\end{lem}}
\newcommand{\bp}{\begin{prop}}
	\newcommand{\ep}{\end{prop}}
\newcommand{\bc}{\begin{cor}}
	\newcommand{\ec}{\end{cor}}

\newcommand{\dint}{\displaystyle\int}
\title{The best Ulam constant of the n-th order linear differential operator with constant coefficients}
\author{Jawad Boutarfass, Iz-iddine EL-Fassi$^{\star}$ and Lahcen Oukhtite
	\\ Department of Mathematics, Faculty of Sciences and Techniques,\\
	S. M. Ben Abdellah University, Fez, Morocco\\ }

\begin{document}
	\renewcommand{\thefootnote}{}
	\footnotetext{$^{\star}$Corresponding author.\\
		E-mail addresses: jawad.boutarfass@usmba.ac.ma (J. Boutarfass),  iziddine.elfassi@usmba.ac.ma; izidd-math@hotmail.fr (Iz. EL-Fassi), lahcen.oukhtite@usmba.ac.ma (L. Oukhtite)}
	\date{}
	\maketitle

\begin{abstract}
	Let $X$ be a Banach space over the field $\mathbb{K} \in \{\mathbb{R}, \mathbb{C}\}$, and let $L:\mathcal{C}^{n}(\mathbb{R},X)\to\mathcal{C}(\mathbb{R},X)$ be the $n$-th order linear differential operator defined by $Ly := y^{(n)} + a_1 y^{(n-1)} + \cdots + a_n y$, where $a_1,\ldots,a_n \in \mathbb{K}$ are fixed scalars and $n$ is a fixed positive integer. It is well known that $L$ is Ulam stable if and only if its characteristic equation has no roots on the imaginary axis. In this paper, we determine the best Ulam constant of $L$ under this condition, for an arbitrary order $n$ and without any restriction on the multiplicities of the roots of the characteristic equation. This improves and completes several known Ulam stability results for this operator and provides the best estimate possible for the error between approximate and exact solutions of the equation $Ly=f$, where $f\in\mathcal{C}(\mathbb{R},X)$. Moreover, as an auxiliary result, we give an explicit formula for the general solution of this equation.
\end{abstract}

	
	
	\section{Introduction}
	\hspace*{3mm}
Differential equations are among the most powerful tools for describing natural and social phenomena, arising in physics, engineering, biology, and economics. This modeling paradigm has given rise to a rich and ever-evolving theory devoted to the study of such equations and their solutions. For general background, we refer to~\cite{Boyce,Braun,Chicone,Coddington,Coddington2,Hartman,Perko,Teschl}.

In many situations (e.g., those arising from the use of numerical methods) only approximate solutions of a differential equation are available, and it is natural to ask how far such approximate solutions may deviate from exact ones. This question lies at the heart of Ulam stability theory, a subject closely connected with several areas of mathematics, including the theories of shadowing, approximation, fixed point, perturbation, and optimization; see, e.g., \cite{Backes,Brzdek-fixedpoint,Ciplinski,Tabor,Yuan}.

Motivated by this question, in this paper we focus on the Ulam stability of linear differential operators with constant coefficients. To be more precise, throughout the paper, let $\mathbb{K}$ denote the field of real or complex numbers, let $X$ be a Banach space over $\mathbb{K}$, and let $n$ be a fixed positive integer. We study the Ulam stability of the $n$-th order linear operator $L:\mathcal{C}^{n}(\mathbb{R},X)\to\mathcal{C}(\mathbb{R},X)$ defined by
\[
Ly:= y^{(n)}
+a_{1}y^{(n-1)}
+\cdots
+a_{n}y,
\qquad
y\in\mathcal{C}^{n}(\mathbb{R},X),
\]
where $a_1,\ldots,a_n\in\mathbb{K}$ are fixed scalars. Here, $\mathcal{C}(\mathbb{R},X)$ denotes the vector space of all continuous functions $f:\mathbb{R}\to X$, and $\mathcal{C}^{n}(\mathbb{R},X)$ denotes the vector subspace of $\mathcal{C}(\mathbb{R},X)$ consisting of all $n$-times differentiable functions whose $n$-th derivative is continuous. Our main objective is to completely solve the problem of determining the best Ulam constant of this operator, and thereby complete the results obtained in~\cite{Baias-RACSAM,Baias-Mathematics,Baias-Axioms,Miura-JMAA-2003,Onitsuka,Takahasi-Korean}.

Let us first recall some basic notions from Ulam stability theory that will be used throughout the paper. We begin with the notion of a gauge on a vector space $E$ over the field $\mathbb{K}$ (see also~\cite{Brzdek1,Miura}).

\begin{dfn}
	A mapping $\rho_E:E\to[0,\infty]$ is said to be a \textit{gauge} on $E$ if:
	\begin{enumerate}
		\item[$\mathrm{(i)}$] $\rho_{E}(x)=0$ if and only if $x=0$;
		\item[$\mathrm{(ii)}$] $\rho_{E}(\lambda x)=|\lambda|\,\rho_{E}(x)$ for all $x\in E$ and $\lambda\in\mathbb{K}\setminus\{0\}$.
	\end{enumerate}
\end{dfn}

The notion of a gauge extends that of a norm and provides a general approach for studying the Ulam stability of operators; see~\cite{Brzdek-gauges,Brzdek1} for further details.

Let $E$ and $F$ be vector spaces over $\mathbb{K}$ equipped with gauges $\rho_{E}$ and $\rho_{F}$, respectively, and let $T:E\to F$ be a linear operator. 

\begin{dfn}[\cite{Miura}]\label{def-stability}
	We say that $T: E\to F$ is \emph{Ulam stable} (or \emph{Hyers--Ulam stable}) if there exists a constant $K\ge 0$ with the following property: for every $\epsilon >0$, every $b\in T(E)$, and every $y\in E$ satisfying $\rho_{F}(Ty -b) \le \epsilon$, there exists $y_0\in E$ such that $Ty_0 =b$ and $\rho_{E}(y-y_0)\le K\epsilon$.
\end{dfn}

An element $y\in E$ satisfying the inequality $\rho_{F}(Ty -b) \le \epsilon$ for some $\epsilon>0$ is called an \textit{approximate solution} of the equation $Tx=b$. Therefore, Definition~\ref{def-stability} can be roughly reformulated as follows: the operator $T$ is Ulam stable if every approximate solution of the equation $Tx=b$ is close to an exact solution. This kind of stability dates back to a 1940 talk by S.\,M. Ulam at the University of Wisconsin, in which he raised this question with regard to group homomorphisms:  \emph{must an approximate homomorphism lie near an exact one?} (see~\cite{Ulam}). A year later, D.\,H. Hyers gave a partial answer to Ulam's question by establishing the stability of Cauchy's additive equation in Banach spaces~\cite{Hyers}. Since then, the notion of Ulam stability has been extensively investigated and extended to numerous classes of equations (e.g., functional, differential, difference, and integral) in a wide range of structures and contexts (see, e.g.,~\cite{Aoki,Baker,Bourgin1,Boutarfass,Ciplinski,Jung-book,Gavruta,Hyers1,Rassias,Yuan}).

The study of the Ulam stability of linear operators was initiated in~\cite{Htori,Miura,Takagi} and has since been developed considerably, yielding results for broad classes of linear operators arising in various areas of mathematics; see~\cite{Brzdek1,Brzdek-gauges,Popa-Aequat,Popa-Appro}. 

Obłoza appears to have been the first author to investigate the Ulam stability of differential equations~\cite{Obloza,Obloza2}.
 A few years later, Alsina and Ger~\cite{Alsina} showed that the real-valued differential equation $y'=y$ is Ulam stable with Ulam constant $K=3$, a result that was soon refined and extended in several works (see~\cite{Miura-JMAA-2003,Onitsuka,Takahasi-Korean,Takahasi-2004}). The most significant advance, however, came with Miura, Miyajima, and Takahasi~\cite{Miura}, who characterized the Ulam stability of $n$-th order linear differential operators with constant coefficients. Since this characterization forms the starting point for our investigation, we recall it below.

Define the map
$\|\cdot\|_{\infty}:\mathcal{C}(\mathbb{R},X)\to[0,\infty]$
by
\[
\|f\|_{\infty}
:=
\sup_{t\in\mathbb{R}}\|f(t)\|,
\qquad
f\in\mathcal{C}(\mathbb{R},X).
\]
Then $\|\cdot\|_{\infty}$ is a gauge. We always assume that $\mathcal{C}(\mathbb{R}, X)$ and $\mathcal{C}^{n}(\mathbb{R}, X)$ are equipped with this gauge. As usual, $\operatorname{Re}(z)$ and $\operatorname{Im}(z)$ stand for the real and imaginary parts of a complex number $z$, respectively.

\begin{thm}[\cite{Miura}]
	\label{thm-Miura}
	The operator $L$ is Ulam stable if and only if its characteristic equation has no roots on the imaginary axis. In this case, for every $\varepsilon>0$, every $f\in\mathcal{C}(\mathbb{R},X)$, and every $y\in\mathcal{C}^{n}(\mathbb{R},X)$ satisfying $\|Ly -f\|_{\infty}\le\varepsilon$, there exists a unique $x\in\mathcal{C}^{n}(\mathbb{R},X)$ such that $Lx =f$ and
\begin{equation}
		\|y-x\|_{\infty}\le
	\frac{\varepsilon}{\displaystyle\prod_{i=1}^{n}|\operatorname{Re}(r_i)|},\label{const-Miura}
\end{equation}

	where $r_1,\ldots,r_n\in\mathbb{C}$ denote the roots of the characteristic equation of $L$.
\end{thm}
\noindent Generalizations of Theorem~\ref{thm-Miura} can be found in~\cite{Buica,Popa,Popa-Comput,Jung}; for related Ulam stability results on other classes of differential and partial differential equations and operators, we refer to~\cite{Brzdek1,Jung2009,Lungu,Novac}.

The constant $K$ appearing in Definition~\ref{def-stability} is called an \emph{Ulam constant} of $T$. In general, the infimum of the set of all Ulam constants of $T$ need not itself be an Ulam constant (see~\cite{Brzdek1,Popa-Aequat}). When this infimum is itself an Ulam constant, it is called the \emph{best Ulam constant} or simply the \emph{Ulam constant} of $T$, denoted $K_{T}$. To the best of our knowledge, Rassias and Tabor~\cite{Rassias-best} were the first to raise the question of the best Ulam constant for certain functional equations. This question was subsequently taken up by numerous specialists for other classes of equations, as well as for operators (see, e.g.,~\cite{Baias-RACSAM,Baias-Bull,Brzdek1,Jung-book,Miura-JMAA-2003,Onitsuka,Onitsuka-Comput,Popa-JMAA,Popa-Aequat,Takahasi-Korean,Yuan}).

 It is worth noting that determining the best Ulam constant of an operator or an equation is an interesting problem, as it provides the sharpest possible measure for the distance between an approximate solution and the nearest exact solution. Such information is important in practice, since many real-world problems cannot be solved exactly and are instead treated using approximate methods, such as numerical methods. Therefore, knowing how far the resulting approximate solutions may differ from exact ones is essential.

For the operator $L$, the best Ulam constant has been determined only in a few special cases. The simplest such case is $n=1$, where $L$ takes the form $Ly=y'+\lambda y$ for some $\lambda\in\mathbb{K}$; here $L$ is Ulam stable if and only if $\operatorname{Re}(\lambda)\neq0$, with best Ulam constant $K_L=|\operatorname{Re}(\lambda)|^{-1}$~\cite{Miura-JMAA-2003,Onitsuka,Takahasi-Korean}. This constant was determined for $n=2$ in~\cite{Baias-RACSAM} and for $n=3$ in~\cite{Baias-Axioms}. Related results for first- and second-order linear differential equations with certain classes of variable coefficients can be found in~\cite{Anderson,Baias-2026,Fukutaka,Fukutaka-Hillseq}. For general $n>1$, a formula for it has been obtained in~\cite{Baias-Mathematics}, but only under the restrictive assumption that all roots of the characteristic equation are simple (i.e., have multiplicity one). The results in~\cite{Baias-RACSAM,Baias-Axioms} show that the constant in~\eqref{const-Miura} is not optimal in general (see also Corollary~\ref{cor4.1} at the end of this paper).

The main objective of this paper is to complete and unify the aforementioned results from~\cite{Baias-RACSAM,Baias-Mathematics,Baias-Axioms,Miura-JMAA-2003,Onitsuka,Takahasi-Korean}. Namely, we determine the best Ulam constant of $L$ for an arbitrary order $n$, without any restriction on the multiplicities of the roots of the characteristic equation. Moreover, we derive several results on the best Ulam constant for important special cases of $L$. As an auxiliary result, we also provide, in the next section, an explicit formula for the general solution of the nonhomogeneous equation $Ly=f$ ($f\in\mathcal{C}(\mathbb{R},X)$).



\section{Auxiliary results }\label{sect3}
	\hspace*{3mm}
In this section, we derive an explicit formula for the general solution of the nonhomogeneous equation
\begin{equation}
	\label{non-hom}
	Ly = f,
\end{equation}
where $f$ is a fixed function in $\mathcal{C}(\mathbb{R},X)$. This formula will be used in the proof of our main results, presented in the next section.

We first recall the structure of the solution set of the homogeneous equation $Ly=0$. For a positive integer $k$, we write $\mathbb{N}_{k}:=\{1,2,\ldots,k\}$. Let $r_1,\ldots,r_m$ ($m\in\mathbb{N}_{n}$) be the pairwise distinct complex roots of the characteristic equation
\begin{equation}
	r^{n}+a_1r^{n-1}+\cdots+a_n=0,\label{char-equ}
\end{equation}
with respective multiplicities $n_1,\ldots,n_m$, so that $\sum_{k=1}^{m}n_k=n$, and set

$$
\widehat{q}(0)=0,\qquad \widehat{q}(l)=\sum_{j=1}^{l}n_j,\qquad l\in\mathbb{N}_{m-1}.
$$

It is well known (see, e.g., \cite{Coddington2,Miura}) that, in the case $\mathbb{K}=\mathbb{C}$, the $n$ functions
\begin{equation}\label{set-sol}
	x_{\widehat{q}(i-1)+j}(t)=t^{j-1}e^{r_it}, \qquad i\in\mathbb{N}_m,\ j\in\mathbb{N}_{n_i},
\end{equation}
form a fundamental set of solutions of $Ly=0$. Thus, in this case, every $x\in\ker L$ can be written as
\begin{equation}\label{sol-hom-eq}
	x(t)=\sum_{i=1}^{m}\sum_{j=1}^{n_i}t^{j-1}e^{r_it}c_{\widehat{q}(i-1)+j}, \qquad t\in\mathbb{R},
\end{equation}
for some constants $c_1,\ldots,c_n\in X$; see \cite[Lemma 2.7]{Miura}.

The general solution of the nonhomogeneous equation~\eqref{non-hom} can be expressed as the sum of the general solution of the equation $Ly=0$ and a particular solution of~\eqref{non-hom}. In the following theorem, we use the method of variation of parameters to derive a particular solution of~\eqref{non-hom} and thereby obtain an explicit formula for the general solution of~\eqref{non-hom}, which will be useful in what follows. Since we are unaware of any published form of such a result, and given its crucial role in the proof of our main results, we include a proof of it for completeness.

\begin{thm}
	\label{thm-gen-sol-hom}
	Let $\mathbb{K}=\mathbb{C}$. Then $y\in \mathcal{C}^{n}(\mathbb{R},X)$ is a solution of~\eqref{non-hom} if and only if there exist constants $c_{1},\ldots,c_{n}\in X$ such that
	\begin{equation}\label{gen-sol-nonhom}
	y(t)=
	\sum_{i=1}^{m}\sum_{j=1}^{n_i}
	t^{j-1}e^{r_it}c_{\widehat{q}(i-1)+j}
	+
	\sum_{i=1}^{m}\sum_{k=1}^{n_i}
	\frac{\phi_i^{(n_i-k)}(r_i)}
	{(k-1)!(n_i-k)!}
	\int_{0}^{t}
	(t-\tau)^{k-1}e^{r_i(t-\tau)}f(\tau)\,d\tau
	\end{equation}
	for $t\in\mathbb{R}$, where, for each $i\in\mathbb{N}_m$,
	\begin{equation}
		\label{Qi}
		\phi_i(r)=
		\begin{cases}
			\displaystyle
			\prod_{l=1,
				\, l\neq i}^{m}
			(r-r_l)^{-n_l},
			& m>1,\\[3mm]
			1,
			& m=1,
		\end{cases}
	\end{equation}
	for every $r\in\mathbb{C}\setminus\{r_l:\ l\in\mathbb{N}_{m}\setminus\{i\}\}$.
\end{thm}
\begin{proof}
By the method of variation of parameters (see, e.g., \cite{Coddington2}), the nonhomogeneous equation~\eqref{non-hom} admits a particular solution of the form
	\begin{equation}
		\label{solp}
		y_p(t)
		=
		\sum_{i=1}^{m}\sum_{j=1}^{n_i}
		t^{j-1}e^{r_i t}
		c_{\widehat{q}(i-1)+j}(t),
		\qquad t\in\mathbb{R},
	\end{equation}
	where $c_{1},\ldots,c_{n}\in \mathcal{C}^{1}(\mathbb{R},X)$ satisfy
	\begin{equation}
		\label{mat-equ}
		W(t)\bigl(c_{1}'(t),\ldots,c_{n}'(t)\bigr)^{T}
		=
		\bigl(0,\ldots,0,f(t)\bigr)^{T},
		\qquad t\in\mathbb{R},
	\end{equation}
	and $W(t)$ denotes the fundamental matrix associated with the fundamental set of solutions defined by~\eqref{set-sol}, that is, the $n\times n$ matrix whose entries are given by
	\[
	w_{i,\widehat{q}(k-1)+j}(t)
	:=
	\frac{d^{i-1}}{dt^{i-1}}x_{\widehat{q}(k-1)+j}(t)
	=
	\frac{d^{i-1}}{dt^{i-1}}
	\Bigl(t^{j-1}e^{r_{k}t}\Bigr),
	\qquad t\in\mathbb{R},
	\]
	for $i\in\mathbb{N}_{n}$, $k\in\mathbb{N}_{m}$, and $j\in\mathbb{N}_{n_{k}}$.
	
	We now proceed to solve the matrix equation~\eqref{mat-equ}. To this end, let $U(t)$ be the upper triangular $n\times n$ matrix with entries
	\begin{equation}
		\label{defU}
		u_{\widehat{q}(k-1)+i,\,\widehat{q}(l-1)+j}(t):=
		\begin{cases}
			\displaystyle
			\frac{(-t)^{j-i}e^{-r_k t}}{(i-1)!(j-i)!},
			& i\le j,\ k=l,\\[3mm]
			0,
			& \text{otherwise},
		\end{cases}
	\end{equation}
	for $k,l\in\mathbb{N}_{m}$, $i\in\mathbb{N}_{n_{k}}$, and $j\in\mathbb{N}_{n_{l}}$, and define the $n\times n$ matrix
	\begin{equation}
		\label{DefV}
		V(t):=W(t)U(t), \qquad t\in\mathbb{R}.
	\end{equation}
	
	We first compute the entries of $V(t)$. Let $i\in\mathbb{N}_{n}$, $k\in\mathbb{N}_{m}$, and $j\in\mathbb{N}_{n_k}$. Then
	\[
	v_{i,\widehat{q}(k-1)+j}(t)
	=
	\sum_{l=1}^{n}
	w_{i,l}(t)\,u_{l,\widehat{q}(k-1)+j}(t),
	\qquad t\in\mathbb{R}.
	\]
	Since, by~\eqref{defU},
	\[
	u_{l,\widehat{q}(k-1)+j}(t)=0,
	\qquad t\in\mathbb{R},
	\]
	for every $l\in\mathbb{N}_{n}\setminus\{\widehat{q}(k-1)+1,\ldots,\widehat{q}(k-1)+j\}$, it follows that
	\begin{align}
		v_{i,\widehat{q}(k-1)+j}(t)&= \sum_{l=\widehat{q}(k-1)+1}^{\widehat{q}(k-1)+j}
		w_{i,l}(t)\,u_{l,\widehat{q}(k-1)+j}(t) \nonumber\\[2mm]
		&= \sum_{l=1}^{j}
		w_{i,\widehat{q}(k-1)+l}(t)\,u_{\widehat{q}(k-1)+l,\widehat{q}(k-1)+j}(t)\nonumber \\[2mm]
		&= \sum_{l=1}^{j}\frac{d^{i-1}}{dt^{i-1}}
		\Bigl(t^{l-1}e^{r_{k}t}\Bigr)\frac{(-t)^{j-l}e^{-r_{k} t}}{(l-1)!(j-l)!} \nonumber \\[2mm]
		&=	\sum_{l=1}^{j}\sum_{s=0}^{\min(i-1, l-1)}\binom{i-1}{s}\frac{t^{l-1-s} }{(l-1-s)!}\, r_{k}^{i-1-s}\frac{(-t)^{j-l}}{(j-l)!}	
		\nonumber	\\[2mm]
		&=	\sum_{l=1}^{j}\sum_{s=1}^{\min(i, l)}\binom{i-1}{s-1}\frac{(-1)^{j-l} }{(j-l)!(l-s)!}\, r_{k}^{i-s}\, t^{j-s}
		\nonumber	\\[2mm]
		&=	\sum_{s=1}^{\min(i, j)}\binom{i-1}{s-1}r_{k}^{i-s}\, t^{j-s}\sum_{l=s}^{j}\frac{(-1)^{j-l} }{(j-l)!(l-s)!}, \qquad t\in \mathbb{R}. \label{comput}
	\end{align}
	Note that, for every $s\in\mathbb{N}_j$,
	\begin{equation}
		\label{Kronker}
		\sum_{l=s}^{j}\frac{(-1)^{j-l}}{(j-l)!(l-s)!}
		=
		\frac{1}{(j-s)!}
		\sum_{l=0}^{j-s}\binom{j-s}{l}(-1)^{j-s-l}
		=
		\delta_{s,j},
	\end{equation}
	where $\delta_{s,j}$ denotes the Kronecker delta, i.e., $\delta_{s,j}=1$ if $s=j$ and $\delta_{s,j}=0$ otherwise. Combining~\eqref{comput} and~\eqref{Kronker}, we obtain
	\begin{equation*}
		v_{i, \widehat{q}(k-1)+j}(t)=	\sum_{s=1}^{\min(i, j)}\binom{i-1}{s-1}r_{k}^{i-s}\, t^{j-s}\, \delta_{s,j}=\begin{cases}
			\binom{i-1}{j-1}r_{k}^{i-j}, & j\le i,
			\\[2mm]	
			0, & j> i.
		\end{cases}
	\end{equation*}
	Thus, the matrix $V(t)$ does not depend on $t$, and we denote it simply by $V:=V(t)$, $t\in\mathbb{R}$. This matrix is called the \emph{confluent Vandermonde matrix}, and its determinant is given by
	\[
	\det V=
	\begin{cases}
		\displaystyle
		\prod_{1\le k<l\le m}(r_l-r_k)^{n_k n_l},
		& m>1,\\[2mm]
		1,
		& m=1.
	\end{cases}
	\]
	Since $r_1,\ldots,r_m$ are pairwise distinct, $\det(V)\neq 0$, and hence $V$ is invertible. Furthermore, it is well known (see, e.g., \cite{Yuan,Mouçouf}) that the entries of the $n$-th column of $V^{-1}$ are given by
	\begin{equation}
		\label{InverseV}
		v^{*}_{\widehat{q}(i-1)+j,\,n}
		=
		\frac{\phi_i^{(n_i-j)}(r_i)}{(n_i-j)!},
		\qquad
		i\in\mathbb{N}_m,\quad
		j\in\mathbb{N}_{n_i},
	\end{equation}
	where $\phi_i$ is defined by~\eqref{Qi}.
	
Observe also that
\[
\det U(t)=\prod_{k=1}^{m}\prod_{i=1}^{n_k}u_{\widehat{q}(k-1)+i,\,\widehat{q}(k-1)+i}(t)=\prod_{k=1}^{m}\prod_{i=1}^{n_k}\frac{e^{-r_{k}t}}{(i-1)!}, \qquad t\in\mathbb{R},
\]
which shows that $U(t)$ is invertible for every $t\in\mathbb{R}$. Consequently, by~\eqref{DefV}, $W(t)$ is invertible as well, with
\begin{equation}
	\label{Inv-W}
	W(t)^{-1}=U(t)V^{-1}, \qquad t\in\mathbb{R}.
\end{equation}
	
	We now return to~\eqref{mat-equ}. Combining it with~\eqref{Inv-W} yields
	\begin{align*}
		\bigl(c_{1}'(t),\ldots,c_{n}'(t)\bigr)^{T}
		&=
		W(t)^{-1}
		\bigl(0,\ldots,0,f(t)\bigr)^{T}
		\\
		&=
		U(t)V^{-1}
		\bigl(0,\ldots,0,f(t)\bigr)^{T}
		\\
		&=
		U(t)
		\bigl(v^{*}_{1,n},\ldots,v^{*}_{n,n}\bigr)^{T}
		f(t),
		\qquad t\in\mathbb{R}.
	\end{align*}
	Hence,
	\[
	c_{\widehat{q}(i-1)+j}'(t)
	=
	\sum_{k=1}^{n}
	u_{\widehat{q}(i-1)+j,k}(t)\,
	v^{*}_{k,n}\,
	f(t),
	\qquad
	t\in\mathbb{R},
	\]
	for every $i\in\mathbb{N}_{m}$ and $j\in\mathbb{N}_{n_i}$. Since
	\[
	u_{\widehat{q}(i-1)+j,k}(t)=0,
	\qquad
	t\in\mathbb{R},
	\]
	whenever $k\notin\{\widehat{q}(i-1)+j,\ldots,\widehat{q}(i-1)+n_i\}$, it follows from~\eqref{defU} and~\eqref{InverseV} that
	\begin{align*}
		c_{\widehat{q}(i-1)+j}'(t)
		&=
		\sum_{k=j}^{n_i}
		u_{\widehat{q}(i-1)+j,\widehat{q}(i-1)+k}(t)\,
		v^{*}_{\widehat{q}(i-1)+k,n}\,
		f(t)
		\\[2mm]
		&=
		\frac{1}{(j-1)!}
		\sum_{k=j}^{n_i}
		\frac{\phi_i^{(n_i-k)}(r_i)}
		{(k-j)!(n_i-k)!}
		(-t)^{k-j}e^{-r_it}\,
		f(t),
		\qquad
		t\in\mathbb{R},
	\end{align*}
	for $i\in\mathbb{N}_{m}$ and $j\in\mathbb{N}_{n_i}$, and therefore
	\[
	c_{\widehat{q}(i-1)+j}(t)
	=
	\frac{1}{(j-1)!}
	\sum_{k=j}^{n_i}
	\frac{\phi_i^{(n_i-k)}(r_i)}{(k-j)!(n_i-k)!}
	\int_{0}^{t}
	(-\tau)^{k-j}e^{-r_i\tau}f(\tau)\,d\tau
	+u,
	\qquad t\in\mathbb{R},
	\]
	for $i\in\mathbb{N}_{m}$ and $j\in\mathbb{N}_{n_i}$, where $u$ is an arbitrary constant in $X$. Choosing $u=0$, we conclude that equation~\eqref{non-hom} admits the particular solution~\eqref{solp}, with
	\[
	c_{\widehat{q}(i-1)+j}(t)
	=
	\frac{1}{(j-1)!}
	\sum_{k=j}^{n_i}
	\frac{\phi_i^{(n_i-k)}(r_i)}{(k-j)!(n_i-k)!}
	\int_{0}^{t}
	(-\tau)^{k-j}e^{-r_i\tau}f(\tau)\,d\tau,
	\qquad t\in\mathbb{R},
	\]
	for $i\in\mathbb{N}_{m}$ and $j\in\mathbb{N}_{n_i}$.
	
	We now compute $y_p(t)$ explicitly:
	\begin{align}
		\label{comput-yp}
		y_p (t)	&= \sum_{i=1}^{m}\sum_{j=1}^{n_i}
		t^{j-1}e^{r_it}c_{\widehat{q}(i-1)+j}(t)
		\nonumber\\
		& =
		\sum_{i=1}^{m}\sum_{j=1}^{n_i}
		t^{j-1}e^{r_it}\frac{1}{(j-1)!}
		\sum_{k=j}^{n_i}
		\frac{\phi_i^{(n_i-k)}(r_i)}
		{(k-j)!(n_i-k)!}
		\int_{0}^{t}
		(-\tau)^{k-j}e^{-r_i\tau}f(\tau)\,d\tau
		\nonumber\\
		& =
		\sum_{i=1}^{m}\sum_{k=1}^{n_i}
		\frac{\phi_i^{(n_i-k)}(r_i)}
		{(n_i-k)!}
		\int_{0}^{t}
		\sum_{j=1}^{k}
		\frac{t^{j-1}(-\tau)^{k-j}}
		{(j-1)!(k-j)!}
		e^{r_i(t-\tau)}f(\tau)\,d\tau
		\nonumber\\
		& =
		\sum_{i=1}^{m}\sum_{k=1}^{n_i}
		\frac{\phi_i^{(n_i-k)}(r_i)}
		{(k-1)!(n_i-k)!}
		\int_{0}^{t}
		\sum_{j=0}^{k-1}
		\binom{k-1}{j}
		t^{j}(-\tau)^{k-1-j}
		e^{r_i(t-\tau)}f(\tau)\,d\tau
		\nonumber\\
		& =
		\sum_{i=1}^{m}\sum_{k=1}^{n_i}
		\frac{\phi_i^{(n_i -k)}(r_i)}
		{(k-1)!(n_i-k)!}
		\int_{0}^{t}
		(t-\tau)^{k-1}e^{r_i(t-\tau)}f(\tau)\,d\tau,
		\qquad t\in\mathbb{R}.
	\end{align}
	
	Finally, since the general solution of the nonhomogeneous equation~\eqref{non-hom} is given by $y=y_h+y_p$, where $y_h$ denotes the general solution of the homogeneous equation $Ly=0$, it follows from~\eqref{sol-hom-eq} and~\eqref{comput-yp} that $y\in \mathcal{C}^{n}(\mathbb{R},X)$ is a solution of~\eqref{non-hom} if and only if it is of the form~\eqref{gen-sol-nonhom} for some $c_1, \ldots,c_n \in X$.
\end{proof}

Although we shall only need Theorem~\ref{thm-gen-sol-hom} in the proof of our main results, we find it useful to conclude this section by giving an explicit description of the general solution of~\eqref{non-hom} in the real case. This description will be obtained from Theorem~\ref{thm-gen-sol-hom} via the complexification of a real vector space.

Let $\mathbb{K}=\mathbb{R}$. The \emph{Taylor complexification} (or simply the \emph{complexification}) of $X$ (see, e.g.,~\cite{Kadison,Muñoz}) is the complex Banach space $X_{\mathbb{C}}:=X\times X$, equipped with the following vector space operations and norm:
\begin{align}
	(a,b)+(c,d)
	&=(a+c,b+d),\nonumber
	\\
	(\alpha+i\beta)(a,b)
	&=(\alpha a-\beta b,\beta a+\alpha b),\label{mult-complex}
	\\
	\|(a,b)\|_{T}
	&=
	\sup_{0\le\theta\le2\pi}
	\|(\cos\theta)a-(\sin\theta)b\|,\nonumber
\end{align}
for all $a,b,c,d\in X$ and $\alpha,\beta\in\mathbb{R}$. The norm $\|\cdot\|_T$, called the \emph{Taylor norm}, satisfies
\begin{equation}
	\label{Tay-nor-ineq}
	\max\{\|u\|,\|v\|\}
	\le
	\|(u,v)\|_{T}
	\le
	\|u\|+\|v\|,
	\qquad
	u,v\in X.
\end{equation}
The gauge associated with $\|\cdot\|_{T}$ will be denoted by $\|\cdot\|_{T,\infty}$, i.e.,
\[
\|y\|_{T,\infty}=\sup_{t\in\mathbb{R}}\|y(t)\|_{T}, \qquad y\in\mathcal{C}(\mathbb{R},X_{\mathbb{C}}).
\]
Let $\widetilde{L}\colon\mathcal{C}^{n}(\mathbb{R},X_{\mathbb{C}})\to\mathcal{C}(\mathbb{R},X_{\mathbb{C}})$ be the linear differential operator defined by
\begin{equation*}
	\widetilde{L}y
	=
	y^{(n)}
	+a_{1}y^{(n-1)}
	+\cdots
	+a_{n}y,
	\qquad
	y\in\mathcal{C}^{n}(\mathbb{R},X_{\mathbb{C}}),
\end{equation*}
which we call the \emph{complexification of $L$}. Now let $y_1,y_2\in\mathcal{C}^{n}(\mathbb{R},X)$ and $f_1,f_2\in\mathcal{C}(\mathbb{R},X)$, and set $\widetilde{y}:=(y_1,y_2)$ and $\widetilde{f}:=(f_1,f_2)$. Since $a_1,\ldots,a_n\in\mathbb{R}$, it follows from~\eqref{mult-complex} that $\widetilde{L}\widetilde{y}=\widetilde{f}$ if and only if $Ly_1=f_1$ and $Ly_2=f_2$. In particular, $\widetilde{y}\in\ker\widetilde{L}$ if and only if $y_1,y_2\in\ker L$.

\begin{cor}
	Let $\mathbb{K}=\mathbb{R}$. Then the following assertions hold.
	\begin{enumerate}
		\item[$\mathrm{(i)}$] A function $y\in\mathcal{C}^{n}(\mathbb{R},X)$ is a solution of~\eqref{non-hom} if and only if
		\begin{equation}\label{gen-sol-realcase}
			\begin{aligned}
				y(t) = {}& \sum_{i=1}^{m}\sum_{j=1}^{n_i} t^{j-1} e^{\operatorname{Re}(r_i)t}
				\left(\cos(\operatorname{Im}(r_i)t)\, c_{\widehat{q}(i-1)+j} - \sin(\operatorname{Im}(r_i)t)\, d_{\widehat{q}(i-1)+j}\right) \\[2mm]
				& + \sum_{i=1}^{m}\sum_{k=1}^{n_i}
				\frac{1}{(k-1)!\,(n_i-k)!}
				\int_{0}^{t} (t-\tau)^{k-1} e^{\operatorname{Re}(r_i)(t-\tau)}\, A_{i,k}(t-\tau)\, f(\tau)\, d\tau,
			\end{aligned}
		\end{equation}
		where
	\begin{equation}\label{A_i,k}
		A_{i,k}(s):=\operatorname{Re}\bigl(\phi_{i}^{(n_i-k)}(r_i)\bigr)\cos(\operatorname{Im}(r_i)s) - \operatorname{Im}\bigl(\phi_{i}^{(n_i-k)}(r_i)\bigr)\sin(\operatorname{Im}(r_i)s), \qquad s\in\mathbb{R},
	\end{equation}
		for all $t\in\mathbb{R}$ and some $(c_1,d_1),\ldots,(c_n,d_n)\in X_{\mathbb{C}}$.
		
		\item[$\mathrm{(ii)}$] If $r_i\in\mathbb{R}$ for all $i\in\mathbb{N}_{m}$, then a function $y\in\mathcal{C}^{n}(\mathbb{R},X)$ is a solution of~\eqref{non-hom} if and only if it has the form~\eqref{gen-sol-nonhom} for some $c_1,\ldots,c_n\in X$.
	\end{enumerate}
\end{cor}
\begin{proof}
Suppose that $y\in\mathcal{C}^{n}(\mathbb{R},X)$ is a solution of~\eqref{non-hom}. Define
$\widetilde y\in\mathcal{C}^{n}(\mathbb{R},X_{\mathbb C})$ by
$\widetilde y(t):=(y(t),0)$ for $t\in\mathbb R$. Then $\widetilde y$ satisfies
$\widetilde L\widetilde y=\widetilde f$, where
$\widetilde f(t):=(f(t),0)$ for every $t\in\mathbb R$. Hence, by Theorem~\ref{thm-gen-sol-hom},
\begin{equation}\label{gen-sol-complexcase}
\widetilde y(t)
=
\sum_{i=1}^{m}\sum_{j=1}^{n_i}
t^{j-1}e^{r_it}\widetilde c_{\widehat q(i-1)+j}
+
\sum_{i=1}^{m}\sum_{k=1}^{n_i}
\frac{\phi_i^{(n_i-k)}(r_i)}
{(k-1)!(n_i-k)!}
\int_0^t
(t-\tau)^{k-1}
e^{r_i(t-\tau)}
\widetilde f(\tau)\,d\tau,
\end{equation}

for every $t\in\mathbb R$, where
$\widetilde c_j=(c_j,d_j)\in X_{\mathbb C}$ for
$j\in\mathbb N_n$. Since
\begin{equation*}
	\begin{aligned}	 e^{r_i s}=e^{\operatorname{Re}(r_i)s}\bigl(\cos(\operatorname{Im}(r_i)s)+\textbf{i}\sin(\operatorname{Im}(r_i)s)\bigr),
		\\
		\phi_i^{(n_i-k)}(r_i)\,e^{r_i s}
		=
		e^{\operatorname{Re}(r_i)s}\Bigl(A_{i,k}(s)+\textbf{i}B_{i,k}(s)\Bigr),
	\end{aligned} \qquad s\in \mathbb{R},
\end{equation*}
where $A_{i,k}(s)$ is given by~\eqref{A_i,k}, and
\[
B_{i,k}(s):=\operatorname{Re}\bigl(\phi_i^{(n_i-k)}(r_i)\bigr)\sin(\operatorname{Im}(r_i)s)+\operatorname{Im}\bigl(\phi_i^{(n_i-k)}(r_i)\bigr)\cos(\operatorname{Im}(r_i)s), \qquad s\in \mathbb{R},
\]
it follows from~\eqref{mult-complex} that~\eqref{gen-sol-realcase} holds.
	
	Conversely, suppose that $y$ is of the form~\eqref{gen-sol-realcase} for some
	$\widetilde c_j=(c_j,d_j)\in X_{\mathbb C}$, $j\in\mathbb N_n$. Set
	$\widetilde f(t):=(f(t),0)$ for every $t\in\mathbb R$, and let
	$\widetilde y\in\mathcal C^n(\mathbb R,X_{\mathbb C})$ be defined as in~\eqref{gen-sol-complexcase}. A direct computation shows that $y(t)=\pi_{1}(\widetilde{y}(t))$ for $t\in\mathbb{R}$, where $\pi_{1}:X_{\mathbb{C}}\to X$ denotes the first coordinate projection, that is,
	\begin{equation}\label{projection}
		\pi_{1}(u_{1},u_{2})=u_{1}, \qquad u_{1},u_{2}\in X.
	\end{equation}
	Invoking Theorem~\ref{thm-gen-sol-hom}, we have $\widetilde L\widetilde y=\widetilde f$, and consequently $Ly=f$.
	
		\smallskip
		(ii)  This follows directly from part~(i), since $r_i,\phi_i^{(n_i-k)}(r_i)\in\mathbb{R}$ for all $i\in\mathbb{N}_{m}$ and $k\in\mathbb{N}_{n_i}$.
\end{proof}
	
	\section{ Best Ulam constant}\label{sec4}
	
In this section, we first establish an Ulam stability result for the operator $L$, and then determine its best Ulam constant. We also derive several consequences concerning the best Ulam constant for important special cases of $L$.

As in Section~\ref{sect3}, $r_1, \ldots, r_m \in \mathbb{C}$ ($m\in \mathbb{N}_n$) denote the pairwise distinct roots of the characteristic equation~\eqref{char-equ}, with respective multiplicities $n_1,\ldots,n_m$. Furthermore, we put
\begin{equation}
	\label{sets}
	I_{1}:=\{i\in\mathbb{N}_m:\operatorname{Re}(r_i)>0\}
	\qquad\text{and}\qquad
	I_{2}:=\{i\in\mathbb{N}_m:\operatorname{Re}(r_i)<0\}.
\end{equation}
The following remark gives equivalent characterizations of Ulam stability of $L$ with Ulam constant $K\geq0$ (see~\cite[Proof of Proposition 5]{Buica}).

\begin{rem}\label{rem-equiv}
	The following statements are equivalent:
	\begin{enumerate}
		\item[$\mathrm{(i)}$] $L$ is Ulam stable with Ulam constant $K$;
		\item[$\mathrm{(ii)}$] there exists $f\in\mathcal{C}(\mathbb{R},X)$ such that for every $\epsilon>0$ and every $y\in\mathcal{C}^{n}(\mathbb{R},X)$ satisfying $\|Ly-f\|_{\infty}\leq\epsilon$, there exists $x\in\mathcal{C}^{n}(\mathbb{R},X)$ such that $Lx=f$ and $\|y-x\|_{\infty}\leq K\epsilon$;
		\item[$\mathrm{(iii)}$] for every $\epsilon>0$ and every $y\in\mathcal{C}^{n}(\mathbb{R},X)$ satisfying $\|Ly\|_{\infty}\leq\epsilon$, there exists $x\in\ker L$ such that $\|y-x\|_{\infty}\leq K\epsilon$.
	\end{enumerate}
\end{rem}

	Remark~\ref{rem-equiv} allows us to work with the simpler formulation~$\mathrm{(iii)}$. In particular, to prove that $L$ is Ulam stable with Ulam constant $K$, it suffices to verify~$\mathrm{(iii)}$. Moreover, $K$ is the best Ulam constant of $L$ if and only if it is the smallest constant for which~$\mathrm{(iii)}$ holds.

We shall need the following lemma.

	\begin{lem}\label{lem4.1}
	Let $\operatorname{Re}(r_{i})\neq 0$ for all $i\in \mathbb{N}_m$. If \(x\in \ker L\) satisfies
		$
		\|x\|_\infty<\infty
		$,
		then \(x(t)=0\) for all \(t\in\mathbb{R}\).
	\end{lem}
	\begin{proof}
		The case $\mathbb{K}=\mathbb{C}$ follows directly from~\cite[Lemma~2.6]{Miura}. Assume now that $\mathbb{K}=\mathbb{R}$, and let $X_{\mathbb{C}}$ be the complexification of $X$ introduced in Section~\ref{sect3}. Define $\widetilde{x}\in \mathcal{C}^{n}(\mathbb{R},X_{\mathbb{C}})$ by $\widetilde{x}(t):=(x(t),0)$ for $t\in\mathbb{R}$. Clearly, \(\widetilde{x}\in\ker\widetilde{L}\). Moreover, by~\eqref{Tay-nor-ineq}, $\|\widetilde{x}\|_{T,\infty}\le\|x\|_{\infty}<\infty$. Hence the complex case gives $\widetilde{x}=0$, and consequently $x=0$.
	\end{proof}
	
We now have all the necessary tools to state and prove our stability result. Throughout, we adopt the convention that $\sum_{i\in \emptyset}=0$. For $i\in\mathbb{N}_m$ and $k\in\mathbb{N}_{n_i}$, define
\begin{equation}
	\label{Phi_i,k}
	\Psi_{i,k}(r)
	=
	\frac{\phi_i(r)}
	{(k-1)!(n_i-k)!},
	\qquad
	r\in
	\mathbb{C}\setminus
	\{r_l:\,l\in\mathbb{N}_m\setminus\{i\}\},
\end{equation}
where $\phi_i$ is defined by~\eqref{Qi}.

\begin{thm}
	\label{thm4.1}
	Let $\operatorname{Re}(r_i)\neq0$ for all $i\in\mathbb{N}_m$.
	Then, for every $\epsilon>0$, every $f\in\mathcal{C}(\mathbb{R},X)$, and every $y\in\mathcal{C}^{n}(\mathbb{R},X)$ satisfying $\|Ly-f\|_{\infty}\leq\epsilon$, there exists a unique $x\in\mathcal{C}^{n}(\mathbb{R},X)$ such that $Lx=f$ and
	\begin{equation*}
		\|y-x\|_{\infty}\leq K\epsilon,
	\end{equation*}
	where
	\begin{equation*}
		K=\int_{0}^{\infty}\biggl(\biggl|\sum_{i\in I_{1}}\sum_{k=1}^{n_{i}}\Psi_{i,k}^{(n_i -k)}(r_i)(-s)^{k-1}e^{-r_{i}s}\biggr|+\biggl|\sum_{i\in I_{2}}\sum_{k=1}^{n_{i}}\Psi_{i,k}^{(n_i -k)}(r_i)s^{k-1}e^{r_{i}s}\biggr|\biggr)\, ds.
	\end{equation*}
\end{thm}

\begin{proof}
	Fix $\epsilon>0$ and let $y\in\mathcal{C}^{n}(\mathbb{R},X)$ satisfy $\|Ly\|_{\infty}\le\epsilon$. We show that there exists $x\in\ker L$ such that $\|y-x\|_{\infty}\le K\epsilon$.
	
	We first consider the case $\mathbb{K}=\mathbb{C}$. Set $g:=Ly$. Then $g\in\mathcal{C}(\mathbb{R},X)$ and
	\begin{equation}
		\label{bound-f}
		\|g\|_{\infty}\le\epsilon.
	\end{equation}
	
	By Theorem~\ref{thm-gen-sol-hom}, there exist vectors $c_1,\ldots,c_n\in X$ such that
	\begin{equation}
		\label{form-y}
		y(t)=
		\sum_{i=1}^{m}\sum_{j=1}^{n_i}
		t^{j-1}e^{r_it}c_{\widehat{q}(i-1)+j}
		+
		\sum_{i=1}^{m}\sum_{k=1}^{n_i}
		\Psi_{i,k}^{(n_i-k)}(r_i)
		\int_{0}^{t}
		(t-\tau)^{k-1}e^{r_i(t-\tau)}g(\tau)\,d\tau,
		\qquad t\in\mathbb{R}.
	\end{equation}
	
	Define $x\in\mathcal{C}^{n}(\mathbb{R},X)$ by
	\begin{equation*}
		x(t):=\sum_{i=1}^{m}\sum_{j=1}^{n_i}
		t^{j-1}e^{r_it}c_{\widehat{q}(i-1)+j}
		+ \sum_{i=1}^{m}\sum_{j=1}^{n_i}
		t^{j-1}e^{r_it}c^{*}_{\widehat{q}(i-1)+j},
		\qquad t\in \mathbb{R},
	\end{equation*}
	where, for each $i\in\mathbb{N}_m$ and $j\in\mathbb{N}_{n_i}$,
	\begin{equation}
		\label{define-constants-c}
		c^{*}_{\widehat{q}(i-1)+j}:=
		\begin{cases}
			\displaystyle
			\frac{1}{(j-1)!}
			\sum_{k=j}^{n_i}
			\frac{\phi_i^{(n_i-k)}(r_i)}
			{(k-j)!(n_i-k)!}
			\int_{0}^{\infty}
			(-\tau)^{k-j}e^{-r_i\tau}g(\tau)\,d\tau,
			& i\in I_1,\\[5mm]
			\displaystyle
			-\frac{1}{(j-1)!}
			\sum_{k=j}^{n_i}
			\frac{\phi_i^{(n_i-k)}(r_i)}
			{(k-j)!(n_i-k)!}
			\int_{-\infty}^{0}
			(-\tau)^{k-j}e^{-r_i\tau}g(\tau)\,d\tau,
			& i\in I_2.
		\end{cases}
	\end{equation}
	We first verify that $c^{*}_1,\ldots,c^{*}_n$ are well defined. Fix $i\in I_1$ and $j\in\mathbb{N}_{n_i}$. Since $\operatorname{Re}(r_i)>0$, it follows from~\eqref{bound-f} that
	\[
	\int_{0}^{\infty}
	\bigl\|(-\tau)^{k-j}e^{-r_i\tau}g(\tau)\bigr\|\,d\tau
	\le
	\epsilon
	\int_{0}^{\infty}
	\tau^{k-j}e^{-\operatorname{Re}(r_i)\tau}\,d\tau =\epsilon\frac{(k-j)!}{\operatorname{Re}(r_i)^{k-j+1}}
	<\infty
	\]
	for every $k=j,\ldots,n_i$. Hence the improper integrals defining $c^{*}_{\widehat{q}(i-1)+j}$ converge absolutely, so $c^{*}_{\widehat{q}(i-1)+j}$ is well defined; the case $i\in I_2$ is treated similarly.
	
	Since $x$ has the form described in~\eqref{sol-hom-eq}, we have $x\in\ker L$. Furthermore,
	\begin{align}
		& \sum_{i\in I_1}\sum_{j=1}^{n_i}
		t^{j-1}e^{r_it}c^{*}_{\widehat{q}(i-1)+j}
		\nonumber\\
		& \quad=
		\sum_{i\in I_1}\sum_{j=1}^{n_i}
		t^{j-1}e^{r_it}\frac{1}{(j-1)!}
		\sum_{k=j}^{n_i}
		\frac{\phi_i^{(n_i-k)}(r_i)}
		{(k-j)!(n_i-k)!}
		\int_{0}^{\infty}
		(-\tau)^{k-j}e^{-r_i\tau}g(\tau)\,d\tau
		\nonumber\\
		& \quad=
		\sum_{i\in I_1}\sum_{k=1}^{n_i}
		\frac{\phi_i^{(n_i-k)}(r_i)}
		{(n_i-k)!}
		\int_{0}^{\infty}
		\sum_{j=1}^{k}
		\frac{t^{j-1}(-\tau)^{k-j}}
		{(j-1)!(k-j)!}
		e^{r_i(t-\tau)}g(\tau)\,d\tau
		\nonumber\\
		& \quad=
		\sum_{i\in I_1}\sum_{k=1}^{n_i}
		\frac{\phi_i^{(n_i-k)}(r_i)}
		{(k-1)!(n_i-k)!}
		\int_{0}^{\infty}
		\sum_{j=0}^{k-1}
		\binom{k-1}{j}
		t^{j}(-\tau)^{k-1-j}
		e^{r_i(t-\tau)}g(\tau)\,d\tau
		\nonumber\\
		& \quad=
		\sum_{i\in I_1}\sum_{k=1}^{n_i}
		\Psi_{i,k}^{(n_i-k)}(r_i)
		\int_{0}^{\infty}
		(t-\tau)^{k-1}
		e^{r_i(t-\tau)}
		g(\tau)\,d\tau,
		\qquad t\in\mathbb{R}.
		\label{Exp-I1}
	\end{align}
	Similarly, one has
	\begin{equation}
		\label{Exp-I_2}
		\sum_{i\in I_2}\sum_{j=1}^{n_i}
		t^{j-1}e^{r_it}c^{*}_{\widehat{q}(i-1)+j}
		=
		-\sum_{i\in I_2}\sum_{k=1}^{n_i}
		\Psi_{i,k}^{(n_i-k)}(r_i)
		\int_{-\infty}^{0}
		(t-\tau)^{k-1}
		e^{r_i(t-\tau)}
		g(\tau)\,d\tau,
		\qquad t\in\mathbb{R}.
	\end{equation}
	Combining~\eqref{Exp-I1} and~\eqref{Exp-I_2}, we obtain
	\begin{align}
		x(t)
		={}&
		\sum_{i=1}^{m}\sum_{j=1}^{n_i}
		t^{j-1}e^{r_it}c_{\widehat{q}(i-1)+j}
		+
		\sum_{i\in I_1}\sum_{k=1}^{n_i}
		\Psi_{i,k}^{(n_i-k)}(r_i)
		\int_{0}^{\infty}
		(t-\tau)^{k-1}
		e^{r_i(t-\tau)}
		g(\tau)\,d\tau
		\nonumber\\
		&
		-
		\sum_{i\in I_2}\sum_{k=1}^{n_i}
		\Psi_{i,k}^{(n_i-k)}(r_i)
		\int_{-\infty}^{0}
		(t-\tau)^{k-1}
		e^{r_i(t-\tau)}
		g(\tau)\,d\tau,
		\qquad t\in\mathbb{R}.
		\label{form-x}
	\end{align}
	By~\eqref{form-y} and~\eqref{form-x},
	\begin{align*}
		\|y(t)-x(t)\|
		&=
		\Bigl\|
		-\int_t^\infty
		\sum_{i\in I_1}\sum_{k=1}^{n_i}
		\Psi_{i,k}^{(n_i-k)}(r_i)
		(t-\tau)^{k-1}
		e^{r_i(t-\tau)}
		g(\tau)\,d\tau
		\\
		&\qquad
		+
		\int_{-\infty}^t
		\sum_{i\in I_2}\sum_{k=1}^{n_i}
		\Psi_{i,k}^{(n_i-k)}(r_i)
		(t-\tau)^{k-1}
		e^{r_i(t-\tau)}
		g(\tau)\,d\tau
		\Bigr\|
		\\
		&=
		\Bigl\|
		-\int_0^\infty
		\sum_{i\in I_1}\sum_{k=1}^{n_i}
		\Psi_{i,k}^{(n_i-k)}(r_i)
		(-s)^{k-1}
		e^{-r_is}
		g(t+s)\,ds
		\\
		&\qquad
		+
		\int_0^\infty
		\sum_{i\in I_2}\sum_{k=1}^{n_i}
		\Psi_{i,k}^{(n_i-k)}(r_i)
		s^{k-1}
		e^{r_is}
		g(t-s)\,ds
		\Bigr\|
		\\
		&\le
		\epsilon
		\int_0^\infty
		\Bigl|
		\sum_{i\in I_1}\sum_{k=1}^{n_i}
		\Psi_{i,k}^{(n_i-k)}(r_i)
		(-s)^{k-1}
		e^{-r_is}
		\Bigr|
		\,ds
		\\
		&\qquad
		+
		\epsilon
		\int_0^\infty
		\Bigl|
		\sum_{i\in I_2}\sum_{k=1}^{n_i}
		\Psi_{i,k}^{(n_i-k)}(r_i)
		s^{k-1}
		e^{r_is}
		\Bigr|
		\,ds
		\\[2mm]
		&=K\epsilon, \qquad t\in \mathbb{R}.
	\end{align*}
	This proves that $\|y-x\|_{\infty}\le K\epsilon$ when $\mathbb{K}=\mathbb{C}$.
	
	We now turn to the case $\mathbb{K}=\mathbb{R}$. Let $X_{\mathbb{C}}$ be the complexification of $X$, and let $\widetilde{y}\in\mathcal{C}^{n}(\mathbb{R},X_{\mathbb{C}})$ be given by $\widetilde{y}(t):=(y(t),0)$ for $t\in\mathbb{R}$. By~\eqref{Tay-nor-ineq}, 
	\[
	\|\widetilde{L}\widetilde{y}\|_{T,\infty}=\|(Ly, 0)\|_{T,\infty}\leq\|Ly\|_{\infty}\leq\epsilon.
	\]
	 Hence, by the first part of the proof, there exists $\widetilde{x}\in\ker \widetilde{L}$ such that
	\begin{equation}
		\label{esti-compl}
		\|\widetilde{y}-\widetilde{x}\|_{T,\infty}\le K\epsilon.
	\end{equation}
	Define $x\in\mathcal{C}^{n}(\mathbb{R},X)$ by $x(t):=\pi_{1}(\widetilde{x}(t))$ for $t\in\mathbb{R}$, where $\pi_{1}:X_{\mathbb{C}}\to X$ denotes the first coordinate projection defined in~\eqref{projection}.
	Since $\widetilde x\in\ker\widetilde L$ and $\ker \widetilde L$ consists precisely of the pairs whose coordinates lie in $\ker L$, the projection $x=\pi_1(\widetilde x)$ belongs to $\ker L$. Furthermore, by~\eqref{Tay-nor-ineq} and~\eqref{esti-compl},
	\[
	\|x(t)-y(t)\|\le\|\widetilde{x}(t)-\widetilde{y}(t)\|_{T}\le K\epsilon,
	\qquad t\in\mathbb{R}.
	\]
	
	We have thus shown that for every $\epsilon>0$ and every $y\in\mathcal{C}^{n}(\mathbb{R},X)$ satisfying $\|Ly\|_{\infty}\le\epsilon$, there exists $x\in\ker L$ such that $\|y-x\|_{\infty}\le K\epsilon$. By Remark~\ref{rem-equiv}, it follows that for every $\epsilon>0$, every $f\in\mathcal{C}(\mathbb{R},X)$, and every $y\in\mathcal{C}^n(\mathbb{R},X)$ satisfying $\|Ly-f\|_{\infty}\le\epsilon$, there exists $x\in\mathcal{C}^n(\mathbb{R},X)$ such that $Lx=f$ and $\|y-x\|_{\infty}\le K\epsilon$.
	
	It remains to establish uniqueness. Suppose that $x,z\in\mathcal{C}^n(\mathbb{R},X)$ satisfy $Lx=Lz=f$, $\|y-x\|_{\infty}\le K\epsilon$, and $\|y-z\|_{\infty}\le K\epsilon$. Then $x-z\in\ker L$ and
	\[
	\|x(t)-z(t)\|
	\le
	\|x(t)-y(t)\|+\|y(t)-z(t)\|
	\le
	2K\epsilon
	\]
	for all $t\in\mathbb{R}$. In particular, $\|x-z\|_{\infty}<\infty$, and Lemma~\ref{lem4.1} yields $x=z$.
\end{proof}
	
	We are now ready to determine the best Ulam constant of $L$.
	
	\begin{thm}\label{thm-best-Ulam}
		Let $\operatorname{Re}(r_{i})\neq 0$ for all $i\in \mathbb{N}_m$. Then the best Ulam constant of $L$ is given by
		\begin{equation}\label{best-const}
			K_{L}= \dint_{0}^{\infty}\biggl(\biggl|\sum_{i\in I_{1}}\sum_{k=1}^{n_{i}}\Psi_{i,k}^{(n_i -k)}(r_i)(-s)^{k-1}e^{-r_{i}s}\biggr|+\biggl|\sum_{i\in I_{2}}\sum_{k=1}^{n_{i}}\Psi_{i,k}^{(n_i -k)}(r_i)s^{k-1}e^{r_{i}s}\biggr|\biggr)\, ds.
		\end{equation}
	\end{thm}
	
	\begin{proof}
		By Theorem~\ref{thm4.1}, $K_L$ is an Ulam constant of $L$. To show that it is the \emph{best} Ulam constant, suppose, for the sake of contradiction, that $L$ admits an Ulam constant $K_1<K_L$. We treat the cases $\mathbb{K}=\mathbb{C}$ and $\mathbb{K}=\mathbb{R}$ separately, the latter being derived from the former.
		
		\medskip
		\noindent\textbf{Case $\mathbb{K}=\mathbb{C}$.}
		Put
		\[
		\varphi_{1}(t)
		=
		\sum_{i\in I_{1}}\sum_{k=1}^{n_{i}}
		\Psi_{i,k}^{(n_i -k)}(r_i)(-t)^{k-1}e^{-r_{i}t},
		\qquad
		\varphi_{2}(t)
		=
		\sum_{i\in I_{2}}\sum_{k=1}^{n_{i}}
		\Psi_{i,k}^{(n_i -k)}(r_i)t^{k-1}e^{r_{i}t},
		\qquad t\in\mathbb{R}.
		\]
		Fix $\theta>0$ and $u\in X$ with $\|u\|=1$, and let $\epsilon>0$. Consider the function
		\[
		g(t)
		:=
		\epsilon
		\begin{cases}
			\dfrac{\overline{\varphi_2(-t)}}{|\varphi_2(-t)|+\theta e^{t}}\,u,
			& t\in(-\infty,-\theta],\\[2mm]
		\alpha(t),
			& t\in(-\theta,\theta),\\[2mm]
			\dfrac{-\overline{\varphi_1(t)}}{|\varphi_1(t)|+\theta e^{-t}}\,u,
			& t\in[\theta,\infty),
		\end{cases}
		\]
		where $\alpha:(-\theta,\theta)\to X$ is a linear function chosen so that $g$ is continuous on $\mathbb{R}$ and $\|\alpha(t)\|\le 1$ for all $t\in(-\theta,\theta)$. By construction,
		\begin{equation}
			\label{bound-f-2}
			\|g\|_{\infty}\le\epsilon.
		\end{equation}
		
		According to Theorem~\ref{thm-gen-sol-hom}, the function $y\in \mathcal{C}^{n}(\mathbb{R},X)$ given by
		\begin{equation*}
			y(t)=
			\sum_{i=1}^{m}\sum_{j=1}^{n_i}
			t^{j-1}e^{r_it}(-c^{*}_{\widehat{q}(i-1)+j})
			+
			\sum_{i=1}^{m}\sum_{k=1}^{n_i}
			\Psi_{i,k}^{(n_i-k)}(r_i)
			\int_{0}^{t}
			(t-\tau)^{k-1}e^{r_i(t-\tau)}g(\tau)\,d\tau
		\end{equation*}
	for $t\in \mathbb{R}$,	where $c^{*}_{\widehat{q}(i-1)+j}$ is defined by~\eqref{define-constants-c}, satisfies $Ly=g$. Moreover, by~\eqref{bound-f-2}, $\|Ly\|_{\infty}=\|g\|_{\infty}\le\epsilon$. Since $K_1$ is assumed to be an Ulam constant for $L$, there exists $x\in\ker L$ such that
		\begin{equation}
			\label{estimation-assumtion}
			\|y-x\|_{\infty}\le K_{1}\epsilon.
		\end{equation}
		
		Using~\eqref{Exp-I1} and~\eqref{Exp-I_2}, we rewrite $y$ as
		\begin{align}
			y(t)
			&=
			-
			\sum_{i\in I_1}\sum_{k=1}^{n_i}
			\Psi_{i,k}^{(n_i-k)}(r_i)\int_t^\infty
			(t-\tau)^{k-1}
			e^{r_i(t-\tau)}
			g(\tau)\,d\tau
			\nonumber\\
			&\qquad
			+
			\sum_{i\in I_2}\sum_{k=1}^{n_i}
			\Psi_{i,k}^{(n_i-k)}(r_i)
			\int_{-\infty}^t(t-\tau)^{k-1}
			e^{r_i(t-\tau)}
			g(\tau)\,d\tau
			\nonumber\\
			&=
			-\int_0^\infty
			\sum_{i\in I_1}\sum_{k=1}^{n_i}
			\Psi_{i,k}^{(n_i-k)}(r_i)
			(-s)^{k-1}
			e^{-r_is}
			g(t+s)\,ds
			\nonumber\\
			&\qquad
			+
			\int_0^\infty
			\sum_{i\in I_2}\sum_{k=1}^{n_i}
			\Psi_{i,k}^{(n_i-k)}(r_i)
			s^{k-1}
			e^{r_is}
			g(t-s)\,ds,
			\label{express-y}
		\end{align}
		from which
		\begin{align}
			\|y(t)\|
			&\le
			\epsilon
			\int_{0}^{\infty}
			\Bigl(
			\Bigl|
			\sum_{i\in I_{1}}\sum_{k=1}^{n_{i}}
			\Psi_{i,k}^{(n_i-k)}(r_i)(-s)^{k-1}e^{-r_{i}s}
			\Bigr|
			+
			\Bigl|
			\sum_{i\in I_{2}}\sum_{k=1}^{n_{i}}
			\Psi_{i,k}^{(n_i-k)}(r_i)s^{k-1}e^{r_{i}s}
			\Bigr|
			\Bigr)\,ds
			\nonumber\\
			&=
			\epsilon K_{L},
			\qquad t\in\mathbb{R}.
			\label{bouned-y}
		\end{align}
		Combining~\eqref{estimation-assumtion} and~\eqref{bouned-y} yields
		\[
		\|x(t)\|
		\le
		\|x(t)-y(t)\|+\|y(t)\|
		\le
		\epsilon(K_{1}+K_{L}),
		\qquad
		t\in\mathbb{R},
		\]
		and Lemma~\ref{lem4.1} then gives $x=0$. Consequently, \eqref{estimation-assumtion} reduces to
		\[
		\|y\|_{\infty}\le\epsilon K_{1},
		\]
		and, in particular,
		\begin{equation}
			\label{bouned-y0}
			\|y(0)\|\le\epsilon K_{1}.
		\end{equation}
		
		Evaluating~\eqref{express-y} at $t=0$ gives
		\begin{align}
			y(0)
			&=
			-\int_0^\infty
			\sum_{i\in I_1}\sum_{k=1}^{n_i}
			\Psi_{i,k}^{(n_i-k)}(r_i)
			(-s)^{k-1}
			e^{-r_is}
			g(s)\,ds
			\nonumber
			\\
			&\qquad	+
			\int_0^\infty
			\sum_{i\in I_2}\sum_{k=1}^{n_i}
			\Psi_{i,k}^{(n_i-k)}(r_i)
			s^{k-1}
			e^{r_is}
			g(-s)\,ds
			\nonumber\\
			&=
			\epsilon
			\int_{\theta}^{\infty}
			\frac{|\varphi_1(s)|^{2}}
			{|\varphi_1(s)|+\theta e^{-s}}
			u\,ds
			+
			\epsilon
			\int_{\theta}^{\infty}
			\frac{|\varphi_2(s)|^{2}}
			{|\varphi_2(s)|+\theta e^{-s}}
			u\,ds
			\nonumber \\
			&
			\qquad	+
			\epsilon
			\int_{0}^{\theta}
			\bigl(\varphi_2(s)g(-s)-\varphi_1(s)g(s)\bigr)\,ds.
		\end{align}
		A routine verification shows that
		\[
		\begin{aligned}
			&	\int_{\theta}^{\infty}
			\frac{|\varphi_1(s)|^{2}}
			{|\varphi_1(s)|+\theta e^{-s}}
			u\,ds
			\to
			\int_{0}^{\infty}
			|\varphi_1(s)|\,u\,ds,
			\\[2mm]
			&\int_{\theta}^{\infty}
			\frac{|\varphi_2(s)|^{2}}
			{|\varphi_2(s)|+\theta e^{-s}}
			u\,ds
			\to
			\int_{0}^{\infty}
			|\varphi_2(s)|\,u\,ds,
		\end{aligned}
		\]
		and
		\[
		\int_{0}^{\theta}
		\bigl(\varphi_2(s)g(-s)-\varphi_1(s)g(s)\bigr)\,ds
		\to 0,
		\]
		as $\theta\to 0$. Hence,
		\begin{equation}
			y(0)\to
			\epsilon
			\left(
			\int_{0}^{\infty}|\varphi_{1}(s)|\,ds
			+
			\int_{0}^{\infty}|\varphi_{2}(s)|\,ds
			\right)u
			=
			\epsilon K_{L}u
			\label{limit-y0}
		\end{equation}
		as $\theta\to 0$. Passing to the limit in~\eqref{bouned-y0} gives
		$
		\epsilon K_{L}\le\epsilon K_{1}
		$,
		which contradicts $K_1<K_L$. This shows that no Ulam constant smaller than $K_L$ exists, so $K_L$ is the best Ulam constant of $L$.
		
		\medskip
		\noindent\textbf{Case $\mathbb{K}=\mathbb{R}$.}
		Let $X_{\mathbb{C}}$ denote the complexification of $X$. Fix $\theta>0$ and $u\in X$ with $\|u\|=1$, and set $\widetilde{u}:=(u,0)$. Let $\epsilon>0$ and define
		\[
		\widetilde{g}(t)
		:=
		\epsilon
		\begin{cases}
			\dfrac{\overline{\varphi_2(-t)}}{|\varphi_2(-t)|+\theta e^{t}}\,\widetilde{u},
			& t\in(-\infty,-\theta],\\[2mm]
			\beta(t),
			& t\in(-\theta,\theta),\\[2mm]
			\dfrac{-\overline{\varphi_1(t)}}{|\varphi_1(t)|+\theta e^{-t}}\,\widetilde{u},
			& t\in[\theta,\infty),
		\end{cases}
		\]
		where $\beta:(-\theta,\theta)\to X_{\mathbb{C}}$ is a linear function chosen so that $\widetilde{g}$ is continuous on $\mathbb{R}$ and $\|\beta(t)\|_{T}\le 1$ for all $t\in(-\theta,\theta)$. Then
		\begin{equation}
			\label{bouned-f-Taylor}
			\|\widetilde{g}\|_{T,\infty}\le\epsilon.
		\end{equation}
		
		According to Theorem~\ref{thm-gen-sol-hom}, the function $\widetilde{y}\in \mathcal{C}^{n}(\mathbb{R}, X_{\mathbb{C}})$ given by
		\begin{equation*}
			\widetilde{y}(t)
			=
			\sum_{i=1}^{m}\sum_{j=1}^{n_i}
			t^{j-1}e^{r_it}(-c^{*}_{\widehat{q}(i-1)+j})
			+
			\sum_{i=1}^{m}\sum_{k=1}^{n_i}
			\Psi_{i,k}^{(n_i-k)}(r_i)
			\int_{0}^{t}
			(t-\tau)^{k-1}e^{r_i(t-\tau)}
			\widetilde{g}(\tau)\,d\tau,
			\qquad t\in\mathbb{R},
		\end{equation*}
		where $c^{*}_{\widehat{q}(i-1)+j}$ is as in~\eqref{define-constants-c} with $\widetilde{g}$ in place of $g$, satisfies $\widetilde{L}\widetilde{y}=\widetilde{g}$. Together with~\eqref{bouned-f-Taylor} this gives
		\begin{equation}
			\label{bouned-Lf-taylor}
			\|\widetilde{L}\widetilde{y}\|_{T,\infty}=\|\widetilde{g}\|_{T,\infty}\le \epsilon.
		\end{equation}
		Proceeding as in~\eqref{bouned-y} and~\eqref{limit-y0}, one readily obtains
		\begin{equation}
			\label{bouned-y-Taylor}
			\|\widetilde{y}(t)\|_{T} \le \epsilon K_{L}, \qquad t\in \mathbb{R},
		\end{equation}
		and $\widetilde{y}(0)\to \epsilon K_{L}\widetilde{u}$ as $\theta\to 0$.
		
		Define $y\in \mathcal{C}^{n}(\mathbb{R},X)$ by $y(t)=\pi_{1}(\widetilde{y}(t))$ for $t\in\mathbb{R}$, where $\pi_{1}:X_{\mathbb{C}}\to X$ denotes the first coordinate projection defined in~\eqref{projection}. By~\eqref{Tay-nor-ineq} and~\eqref{bouned-Lf-taylor},
		\begin{equation*}
			\|Ly\|_{\infty}\le \|\widetilde{L}\widetilde{y}\|_{T,\infty}\le \epsilon.
		\end{equation*}
		Since $K_1$ is assumed to be an Ulam constant for $L$, there exists $x\in \ker L$ such that
		\begin{equation}
			\label{estimation-y-real}
			\|y-x\|_{\infty}\le \epsilon K_{1}.
		\end{equation}
		Moreover, by~\eqref{Tay-nor-ineq} and~\eqref{bouned-y-Taylor},
		\[
		\|y(t)\|
		\le
		\|\widetilde{y}(t)\|_{T}
		\le
		\epsilon K_{L},
		\qquad
		t\in\mathbb{R},
		\]
		so that
		\[
		\|x(t)\|\le \|x(t)-y(t)\|+ \|y(t)\| \le \epsilon(K_1 +K_L), \qquad t\in \mathbb{R}.
		\]
		Lemma~\ref{lem4.1} then gives $x=0$. Consequently, \eqref{estimation-y-real} reduces to $\|y\|_{\infty}\le \epsilon K_1$, and in particular
		\begin{equation}
			\label{bouned-y0-real}
			\|y(0)\| \le \epsilon K_1.
		\end{equation}
		
		On the other hand,
		\[
		\|y(0)-\epsilon K_{L}u\|
		\le
		\|\widetilde{y}(0)-\epsilon K_{L}\widetilde{u}\|_{T}
		\to 0
		\]
		as $\theta\to 0$, so that $\|y(0)\|\to\epsilon K_{L}$. Passing to the limit in~\eqref{bouned-y0-real} gives
		$
		\epsilon K_{L}\le\epsilon K_{1}
		$,
		which contradicts the assumption $K_1<K_L$. This completes the proof.
	\end{proof}
	
	Theorem~\ref{thm-best-Ulam} yields several important consequences for the best Ulam constants of special classes of linear differential operators with constant coefficients, extending and unifying several earlier known results, including those in~\cite{Baias-RACSAM,Baias-Mathematics,Miura-JMAA-2003,Onitsuka,Takahasi-Korean}.
	\begin{cor}\label{cor4.1}
		Let \(\operatorname{Re}(r_i)\neq 0\) for \(i\in \mathbb{N}_{m}\), and let \(K_L\) denote the best Ulam constant of \(L\). Then the following assertions hold.
		\begin{enumerate}
			\item[$\mathrm{(i)}$] If $m=1$, then $K_L =|\operatorname{Re}(r_1)|^{-n}$.
			\item[$\mathrm{(ii)}$] If $m=n$, then 
			\begin{equation}\label{best-const-m=n}
				K_L = \dint_{0}^{\infty}\Bigl(\Bigl|\sum_{i\in I_1}\phi_{i}(r_i)e^{-r_{i}s}\Bigr|+\Bigl|\sum_{i\in I_2}\phi_{i}(r_i)e^{r_{i}s}\Bigr|\Bigr) ds.
			\end{equation}
			\item[$\mathrm{(iii)}$] If $n=m=2$, then 
			\begin{equation*}
				K_L
				=
				\frac{1}{|r_1-r_2|}
				\begin{cases}
					\dint_{0}^{\infty}\bigl|e^{-r_{1}s}- e^{-r_{2}s}\bigr|\, ds, & \operatorname{Re}(r_1)>0,\ \operatorname{Re}(r_2)>0, 
					\\[4mm]
					\dint_{0}^{\infty}\bigl|e^{r_{1}s}- e^{r_{2}s}\bigr|\, ds, & \operatorname{Re}(r_1)<0,\ \operatorname{Re}(r_2)<0, 
					\\[4mm]
					\left|\frac{1}{\operatorname{Re}(r_1)}-\frac{1}{\operatorname{Re}(r_2)}\right|, & \operatorname{Re}(r_1)\operatorname{Re}(r_2)<0.
				\end{cases}
			\end{equation*}
		\end{enumerate}
	\end{cor}
	\begin{proof}
		By Theorem~\ref{thm-best-Ulam}, the best Ulam constant of $L$ is given by formula~\eqref{best-const}.
		
		\smallskip
		\noindent(i) Let $m=1$. Then $n_1=n$, and by~\eqref{Qi} and~\eqref{Phi_i,k},
		\begin{equation}
			\label{Phi_1,k}
			\Psi_{1,k}(r)=\frac{1}{(k-1)!\,(n-k)!},
			\qquad k \in \mathbb{N}_{n}, \quad r\in \mathbb{C}.
		\end{equation}
		Suppose first that $\operatorname{Re}(r_1)>0$, so that $I_1=\{1\}$ and $I_2=\emptyset$. Since $\Psi_{1,k}^{(n-k)}(r_1) = 0$ for every $k = 1, \ldots, n-1$, it follows from~\eqref{best-const} and~\eqref{Phi_1,k} that
		\[
		K_L
		= \dint_{0}^{\infty}\bigl|\Psi_{1,n}(r_1)(-s)^{n-1}e^{-r_{1}s}\bigr|\,ds
		= \frac{1}{(n-1)!}\int_{0}^{\infty}s^{n-1}e^{-\operatorname{Re}(r_1) s}\,ds
		=( \operatorname{Re}(r_1))^{-n}.
		\]
		The case $\operatorname{Re}(r_1)<0$ is treated in the same way, yielding
		\[
		K_L = (-\operatorname{Re}(r_1))^{-n}=|\operatorname{Re}(r_1)|^{-n}.
		\]
		
		\smallskip
		\noindent(ii) Let $m=n$. Then $n_i=1$,  and
		$	\Psi_{i,1}(r_i)=\phi_i(r_i)
		$ for all $i\in\mathbb{N}_{n}$.
		Substituting into~\eqref{best-const} gives~\eqref{best-const-m=n}.
		
		\smallskip
		\noindent(iii) This follows from part~(ii) by a direct computation.
	\end{proof}
	
Note that assertion~(i) of Corollary~\ref{cor4.1} was established only for $n=1$ in~\cite{Miura-JMAA-2003,Onitsuka,Takahasi-Korean}, for $n=2$ in~\cite{Baias-RACSAM}, and for $n=3$ in~\cite{Baias-Axioms}.
 Assertion~(iii) was established in~\cite{Baias-RACSAM} for \(\mathbb{K}=\mathbb{C}\). If \(m=n>1\), then it is straightforward to verify (see~\cite[Remark~4.1]{Yuan}) that, for every \(i\in\mathbb{N}_n\),
	\begin{equation*}
		\phi_i(r_i)
		=
	\displaystyle\prod_{l=1, l\neq i}^{n}(r_i-r_l)^{-1}
		=
		(-1)^{\,n-i}\frac{v_i}{v},
	\end{equation*}
	where \(v=\det(V)\), \(V=V(r_1,\ldots,r_n)\) is the Vandermonde matrix of order \(n\), \(v_i=\det(V_i)\), and \(V_i=V(r_1,\ldots,r_{i-1},r_{i+1},\ldots,r_n)\) is the Vandermonde matrix of order \(n-1\). Substituting this identity into~\eqref{best-const-m=n}, we recover the formula obtained in~\cite{Baias-Mathematics}, i.e.,  
	\begin{align*}
		K_L =&\dint_{0}^{\infty}\biggl(\biggl|\sum_{i\in I_1}(-1)^{n-i}\frac{v_i}{v}e^{-r_{i}s}\biggr|+\biggl|\sum_{i\in I_2}(-1)^{n-i}\frac{v_i}{v}e^{r_{i}s}\biggr|\biggr) ds
		\\[2mm]
		=&\frac{1}{|v|}\dint_{0}^{\infty}\biggl(\biggl|\sum_{i\in I_1}(-1)^{i}v_{i}e^{-r_{i}s}\biggr|+\biggl|\sum_{i\in I_2}(-1)^{i}v_{i}e^{r_{i}s}\biggr|\biggr) ds.
	\end{align*}  
	
	If all the roots of the characteristic equation~\eqref{char-equ} are real, the best Ulam constant of the operator \(L\) admits the following simple explicit formula.
	\begin{prop}\label{prop-best-real}
		Let $r_i \in \mathbb{R}\setminus \{0\}$ for all $i\in \mathbb{N}_{m}$. Then the best Ulam constant of $L$ is given by 
		\begin{equation}\label{best-con-real}
			K_L = \prod\limits_{i=1}^{m}|r_i|^{-n_i}=|a_n|^{-1}.
		\end{equation}
	\end{prop}
	\begin{proof}
		Observe first that
		$
		a_n = (-1)^{n} \prod_{i=1}^{m} r_i^{n_i} \neq 0
		$.
		It was shown in~\cite{Miura} that $K_L = \prod_{i=1}^{m}|r_i|^{-n_i}$ is an Ulam constant for $L$ when $\mathbb{K}=\mathbb{C}$. An argument analogous to the one used in the proof of Theorem~\ref{thm4.1} shows that $K_L$ remains an Ulam constant for $L$ in the real case $\mathbb{K}=\mathbb{R}$.
		
		It remains to show that $K_L$ is the \emph{best} Ulam constant of $L$. Suppose, for contradiction, that $L$ admits an Ulam constant $K<K_L$. Fix $\epsilon>0$ and $u\in X$ with $\|u\|=1$, and set
		\begin{equation}
			\label{def-y}
			y(t):=\epsilon\, a_n^{-1} u, \qquad t\in\mathbb{R}.
		\end{equation}
		Then $\|Ly\|_{\infty}=\epsilon$, so by the assumption that $K$ is an Ulam constant for $L$, there exists $x\in\ker L$ such that
		\begin{equation}
			\label{estimation-prop}
			\|y-x\|_{\infty}\le \epsilon K.
		\end{equation}
		Combining~\eqref{def-y} and~\eqref{estimation-prop} gives, for every $t\in\mathbb{R}$,
		\[
		\|x(t)\| \le \|x(t)-y(t)\| + \|y(t)\| \le \epsilon\bigl(K+|a_n|^{-1}\bigr),
		\]
		and Lemma~\ref{lem4.1} then forces $x=0$. Consequently, \eqref{estimation-prop} reduces to
		\[
		\|y\|_{\infty}\le \epsilon K.
		\]
		On the other hand, \eqref{def-y} gives $\|y\|_{\infty}=\epsilon|a_n|^{-1}=\epsilon K_L$, and therefore
		$
		\epsilon K_L \le \epsilon K,
		$
		which contradicts the assumption $K<K_L$. 
	\end{proof}
	As an application, combining Theorem~\ref{thm-best-Ulam} and Proposition~\ref{prop-best-real} yields the following identity:
	\begin{cor}
		Let $r_i \in \mathbb{R}\setminus \{0\}$ for $i\in \mathbb{N}_{m}$. Then 
		\begin{equation}	\label{eleg-result}
		\dint_{0}^{\infty}\biggl(\biggl|\sum_{i\in I_{1}}\sum_{k=1}^{n_{i}}\Psi_{i,k}^{(n_i-k)}(r_i)(-s)^{k-1}e^{-r_{i}s}\biggr|+\biggl|\sum_{i\in I_{2}}\sum_{k=1}^{n_{i}}\Psi_{i,k}^{(n_i-k)}(r_i)s^{k-1}e^{r_{i}s}\biggr|\biggr)\, ds = \prod\limits_{i=1}^{m}|r_i|^{-n_i}.
		\end{equation}
	\end{cor}
\begin{proof}
	For pairwise distinct roots $r_1, \ldots , r_m \in \mathbb{R}\setminus \{0\}$ of the characteristic equation~\eqref{char-equ} with respective multiplicities $n_1, \ldots, n_m$, the best Ulam constant of $L$ is given both by formula~\eqref{best-const} of Theorem~\ref{thm-best-Ulam} and by formula~\eqref{best-con-real} of Proposition~\ref{prop-best-real}. Combining these two expressions gives~\eqref{eleg-result}.
\end{proof}
	
	\section{Conclusion}
	
	In this paper, we have completely determined the best Ulam constant of the $n$-th order linear differential operator $L$ with constant coefficients. This extends and completes several known results in the literature, where the best Ulam constant of $L$ had previously been obtained only for the particular cases $n=1,2,3$ and, for general $n$, only under the restrictive assumption that all the roots of the characteristic equation of $L$ are simple. Our result provides the optimal estimate of the difference between an approximate solution of the differential equation associated with the operator $L$ and an exact solution. Consequently, it can be applied to the study of perturbations of dynamical systems governed by such an equation, with potential applications in various branches of science, including physics, engineering, biology, and economics.
	
	\section*{Contributions}
	All authors reviewed the manuscript.
	\section*{Conflicts of interest}
	The authors declare that they have no conflicts of interest for this manuscript.\\
	
	\noindent\textbf{Availability of data and material:} Not applicable.\\
	
	\noindent\textbf{Funding:} This research received no funding.
	
	\bibliographystyle{IEEEtran}

\end{document}